\documentclass[preprint, 1p]{elsarticle}
\usepackage{amsthm, amsmath, amssymb}
\usepackage{graphicx}

\newtheorem{theorem}{Theorem}[section]

\newtheorem{lemma}{Lemma}[section]
\newtheorem{proposition}{Proposition}[section]
\newtheorem{definition}{Definition}[section]

\usepackage{color}

\begin{document}
\title{Numerical Solution of the Two-Phase Obstacle Problem by Finite Difference Method}
% Short title for running heads:

\author[inst]{Avetik  Arakelyan}
\ead{arakelyanavetik@gmail.com}
\address[inst]{Institute of Mathematics, NAS of Armenia, Baghramyan 24b, 0019 Yerevan, Armenia}

\author[inst]{Rafayel  Barkhudaryan \fnref{fn1,fn2}\corref{cor1}}
\ead{rafayel@instmath.sci.am}

\author[yer]{Michael  Poghosyan \fnref{fn1}}
\ead{michael@ysu.am}
\address[yer]{Department of Mathematics and Mechanics, Yerevan State University, Alex Ma\-noo\-gian 1, 0025, Yerevan, Armenia}

\fntext[fn1]{R.~Barkhudaryan and M.~Poghosyan would like to thank \emph{G\"{o}ran Gustafssons Foundation} for visiting appointment to KTH.}
\fntext[fn2]{R.~Barkhudaryan acknowledges support by NFSAT, CRDF Global, YSSP Grant no. YSSP-13-24.}
\cortext[cor1]{Corresponding author}

\begin{abstract}
% Body of abstract:
{
In this paper we consider the numerical approximation of the \textit{two-phase membrane (obstacle) problem} by finite difference method. First, we introduce the notion of viscosity solution for the problem and construct certain discrete nonlinear approximation system. The existence and uniqueness  of the solution of the discrete  nonlinear system is proved. Based on that scheme, we propose projected Gauss-Seidel algorithm and prove its convergence. At the end of the paper we present some numerical simulations.

%We propose an algorithm to solve the \textit{two-phase obstacle problem} by finite difference
%method. We prove the existence and uniqueness of the solution of the discrete nonlinear system and obtain an error estimate for finite difference approximation. Also we prove the convergence of the proposed numerical algorithm. At the end of the paper we present some numerical simulations.
}
% Keywords:
\end{abstract}


%\begin{AMS}35R35, 65N06\end{AMS}

\maketitle

\section{Introduction}
\paragraph{\textbf{The Mathematical Setting of the Problem}} Let $\Omega\subset\mathbb R^n$, $n\ge 1$, be a bounded open subset with Lipschitz-regular boundary. Let $g:\partial \Omega\to\mathbb R$ be a continuous function taking both positive and negative values over $\partial\Omega$, and $\lambda^+, \lambda^-:\Omega\to\mathbb R$ are Lipschitz-continuous functions satisfying
$$
\lambda^+ (x)\ge 0,\quad \lambda^-(x)\ge 0,\quad \mbox{and}\quad \lambda^+(x)+\lambda^-(x)>0,\quad x\in \Omega.
$$

The \textit{two-phase obstacle problem}, or the \textit{two-phase membrane problem},  is the problem of  minimization of the cost functional
\begin{equation}\label{eq-1}
\mathcal J(v):=\int_\Omega \left[ \frac{1}{2} |\nabla v|^2 +\lambda^+\max (v,0)+\lambda^-\max (-v,0)\right] dx
\end{equation}
over the set of admissible ``deformations'' $\mathbb K:=\{ v\in H^1(\Omega):\ v-g \in H^1_0(\Omega)\}$.

It is straightforward to see that $\mathcal J$ is coercive, convex and lower-semicontinuous over $H^1(\Omega)$, resulting in the existence of the unique minimum point $u$ of the functional on the affine subspace $\mathbb K\subset H^1(\Omega)$.

Writing down the Euler-Lagrange equation for the minimization problem for the energy functional \eqref{eq-1}, we obtain
\begin{equation}\label{2D-Problem-DifForm}
\left\{
\begin{array}{ll}
\Delta u = \lambda^+\cdot\chi_{\{u>0\}}-\lambda^-\cdot\chi_{\{u<0\}}, & x\in\Omega,\\
u=g, & x\in\partial \Omega,
\end{array}
\right.
\end{equation}
where $\chi_A$ stands for the characteristic function of the set $A$. It is easy to see (cf. \cite{MR1825655}), that the solution (in the weak sense) of \eqref{2D-Problem-DifForm} must coincide with the minimizer $u\in\mathbb K$ of \eqref{eq-1}.

Problem \eqref{2D-Problem-DifForm} is an example of a free boundary problem. Roughly speaking, we need to find a function $u$ satisfying $\Delta u=\lambda^+$ on the set $\{u>0\}$ and $\Delta u=-\lambda^-$ on $\{u<0\}$ and which is $C^{1,\alpha}$ across $\partial\{u>0\}\cup\partial\{u<0\}$. The sets $\{u>0\}$ and $\{u<0\}$, the two {\em phases} for this problem, are not known a priori, and need to be determined along with the solution $u$. So the free boundary for this problem consist of two parts- $\partial \{u>0\}\cap\Omega$ and $\partial  \{u<0\}\cap\Omega$.

\paragraph{\textbf{Physical interpretation and known results}} The problem of minimization of the functional \eqref{eq-1} arises in connection with describing the equilibrium state of a hanging membrane in the two-phase matter with different gravitation densities (say, in water and air), assuming the membrane is fixed on the boundary of a given domain. If the density of the membrane is between the densities of two matters, then the membrane is being buoyed up in the phase with higher density and pulled down in the phase with lower density, and the equilibrium state is described by minimization of the energy functional \eqref{eq-1}. In that case $\lambda^+$ is proportional to the difference between the densities of high-density matter and membrane, and $\lambda^-$ is proportional to the difference between the densities of membrane and low-density matter.

In the case of nonnegative  $g$, one can prove that $u\ge 0$ over $\Omega$, resulting $u$ to be the solution of \textit{one-phase obstacle problem} or the \textit{classical obstacle problem}, which has been extensively studied in the literature. Here we assume that $g$ takes both positive and negative values across the boundary, forcing our problem to have two phases.

The \textit{two-phase obstacle problem} \eqref{2D-Problem-DifForm} has been studied from different viewpoints. As it has been mentioned above, the existence of minimizers is straightforward and is obtained by the direct methods of calculus of variations. The optimal $C^{1,1}_{loc}$ regularity for the solution to \eqref{2D-Problem-DifForm} has been proved in \cite{MR1906034} for constant coefficients $\lambda^\pm$, and the result was extended in \cite{MR1934623} for Lipschitz-regular  $\lambda^\pm$ and in \cite{MR2569898} for H\"older-regular $\lambda^\pm$. The regularity and the geometry of the free boundary has been studied in \cite{MR2340105}, \cite{MR2258264},  \cite{MR2237208}.

Concerning the numerical solution of  the two-phase obstacle problem, in his recent paper \cite{Farid} Bozorgnia discussed three algorithms for numerical solution of two-phase obstacle problem. The first algorithm constructs an iterative sequence converging towards the solution. The second algorithm uses the regularization method to construct an approximation for the solution, and the third is based on Finite Element Method. But here the first and the third methods lack of convergence proofs, and only for the second method the estimates for the difference between the regularized solutions and exact solution are given.

Our main aim in this paper is to construct a Finite Difference approximation for the two-phase obstacle problem and to prove the convergence of the proposed algorithm.

In this paper we use the regularization method to obtain a smooth approximation for two-phase obstacle problem, approximate the latter by Finite Difference Scheme (FDS), and solve the obtained nonlinear system by means of PGS (Projected Gauss-Seidel) method.

\section{Construction of the finite difference scheme}

We start this section by recalling the definition of the viscosity solutions of fully nonlinear second order elliptic differential equations, then we give the reformulation of the differential equation in \eqref{2D-Problem-DifForm} as fully nonlinear equation, which we will refer to as the \textit{Min-Max form of the two-phase obstacle problem}. Using this representation, in the last subsection we construct the corresponding Finite Difference Scheme and prove the existence and uniqueness of the solution to this discrete problem.

\subsection{Degenerate elliptic equations and viscosity solutions}
Let $\Omega$ be an open subset of $\mathbb{R}^n$, and for twice differentiable function $u:\Omega\to\mathbb R$ let $Du$ and $D^2 u$ denote the gradient and Hessian matrix of $u$, respectively. Also let the function $F(x, r, p, X)$ be a continuous real-valued function defined on $\Omega\times\mathbb{R}\times\mathbb{R}^n\times S^n$, with $S^n$ being the space of real symmetric $n\times n$ matrices. Denote
$$
\mathcal F[u](x)\equiv F\left(x, u(x), Du(x), D^2u(x)\right).
$$

We consider the following second order fully nonlinear partial differential equation:
\begin{equation}\label{PDE}
\mathcal F[u](x) = 0, \qquad   x\in\Omega.
\end{equation}

\begin{definition}
The equation \eqref{PDE} is \textbf{degenerate elliptic} if
$$
 F(x, r, p,X)\leq F(x, s, p, Y )  \quad\mbox{whenever }\quad  r\leq s  \quad\mbox{and }\quad Y \leq X,
$$
where $Y \leq X$ means that $X - Y$ is a nonnegative definite symmetric matrix.
\end{definition}

\begin{definition}
$u:\Omega\to\mathbb R$ is called a \textbf{viscosity subsolution} of \eqref{PDE}, if it is upper semicontinuous and for each $\varphi\in C^2(\Omega)$ and local maximum point $x_0\in \Omega$ of $u-\varphi$ we have
\begin{equation}\label{2D-def-Visc-Subsol}
F\left(x_0, u(x_0), D\varphi(x_0), D^2 \varphi(x_0)\right)\le 0.
\end{equation}
\end{definition}

\begin{definition}
$u:\Omega\to\mathbb R$ is called a \textbf{viscosity supersolution} of \eqref{PDE}, if it is lower semicontinuous and for each $\varphi\in C^2(\Omega)$ and local minimum point $x_0\in \Omega$ of $u-\varphi$ we have
$$
F\left(x_0, u(x_0), D\varphi(x_0), D^2 \varphi(x_0)\right)\ge 0.
$$
\end{definition}
\begin{definition}
$u:\Omega\to\mathbb R$ is called a \textbf{viscosity solution} of \eqref{PDE}, if it is both a viscosity subsolution and supersolution (and hence continuous) for \eqref{PDE}.
\end{definition}

The notion of viscosity solution was first introduced in 1981 by Crandall and Lions (see \cite{MR690039} and \cite{MR732102}) for first order Hamilton-Jacobi equations. It turns out that this notion is an effective tool also in the study of second order (elliptic and parabolic) fully nonlinear problems. There is a vast literature devoted to viscosity solutions by now, and for a general theory the reader is referred to \cite{MR1118699}, \cite{MR1351007} and references therein.

\subsection{Min-Max reformulation of the problem}

Now we consider the following nonlinear problem, which we will refer as the \textit{Min-Max form of the two-phase obstacle problem}:

\begin{equation}\label{2D-Problem-Viscosity}
\left\{
\begin{array}{ll}
\min\left(-\Delta u + \lambda^+ ,\max(-\Delta u-\lambda^-,u)\right)=0,& \mbox{in}\ \Omega\\
u=g,& \mbox{on}\ \partial\Omega.
\end{array}
\right.
\end{equation}

If we introduce a function $F:\Omega\times\mathbb{R}\times\mathbb{R}^n\times S^n\to\mathbb R$ by
$$
F(x, r, p,X)=\min(-trace(X) + \lambda^+ ,\max(-trace(X)-\lambda^-,r)),
$$
then the equation in \eqref{2D-Problem-Viscosity} can be rewritten as
\begin{equation}\label{2D-F}
\mathcal F[u](x)=F(x, u, Du, D^2 u)=0\quad \mbox{in}\ \ \Omega,
\end{equation}
and by \textbf{solution to \eqref{2D-Problem-Viscosity}} we mean a function $u\in C(\overline\Omega)$ which is a viscosity solution to \eqref{2D-F} in the sense defined above and satisfies $u=g$ along the boundary $\partial\Omega$.

First we prove the following simple

\begin{lemma}\label{2D-lem-degelliptic}
The equation \eqref{2D-F} is degenerate elliptic.
\end{lemma}
\begin{proof} Let $X,Y\in S^n$ and $r,s\in\mathbb R$ satisfy $Y \leq X $ and $r\leq s$. Then
\begin{multline*}
-trace(X)+\lambda^+\leq-trace(Y)+\lambda^+, \quad {\rm and}\\ \max(-trace(X)-\lambda^-,r)\le \max(-trace(Y)-\lambda^-,s).
\end{multline*}
Therefore
\begin {align*}
F(x, r, p,X) &=\min(-trace(X) + \lambda^+ ,\max(-trace(X)-\lambda^-,r))\\
             &\leq\min(-trace(Y) + \lambda^+ ,\max(-trace(Y)-\lambda^-,s))=F(x, s, p,Y).
\end{align*}
\end{proof}

The next Proposition shows the connection between problems \eqref{2D-Problem-Viscosity} and \eqref{2D-Problem-DifForm}.
%The first part of this proposition can be easily verified by using corresponding definitions, and The second part can be found in %\cite{MR1825655}.

\begin{proposition} If $u$ is the solution (in the weak sense) to \eqref{2D-Problem-DifForm}, then it is a viscosity solution to \eqref{2D-Problem-Viscosity}. Moreover, $u$ satisfies \eqref{2D-Problem-Viscosity} a.e.
\end{proposition}

\begin{proof} Let $u$ be a weak solution of the two-phase obstacle problem \eqref{2D-Problem-DifForm} (we refer to \cite{MR1825655} for the definition of the weak solution). Then  $u$ satisfies the following inequality in the sense of distributions
$$
-\lambda^-\leq\Delta u \leq\lambda^+\quad {\rm in}\quad \Omega,
$$
and hence, the same inequality will be true also in the viscosity sense (see \cite{MR1341739}), in the sense that $u$ is a viscosity subsolution for
the equation $-\Delta v-\lambda^-=0$ and viscosity supersolution for $-\Delta v +\lambda^+=0$.

Let $x_0\in\Omega$ and $\varphi\in C^2(\Omega)$ are such that $x_0$ is a local maximum point of $u-\varphi$. To verify \eqref{2D-def-Visc-Subsol}, we consider two different cases:
\begin{itemize}
\item{$x_0\in\{u>0\}\cup\{u<0\}$}. In this case the solution will be $C^2$ smooth in some neighborhood of $x_0$, and it will satisfy \eqref{2D-Problem-DifForm} in a classical sense. So if we assume, without loss of generality, that $x_0\in\{u>0\}$, then we'll have
    $$-\Delta u (x_0) +\lambda^+(x_0)=0$$
    in a classical sense. On the other hand, by our assumption,
    $$
    \max(-\Delta u(x_0)-\lambda^-(x_0),u(x_0))>0,
    $$
    so
    \begin{multline*}
    F(x_0, u(x_0), Du(x_0), D^2u(x_0))=\\
    \min(-\Delta u (x_0) +\lambda^+(x_0),\max(-\Delta u(x_0)-\lambda^-(x_0),u(x_0)))=0.
    \end{multline*}
    Now, since $x_0$ is a local maximum point of $u-\varphi$, and $u-\varphi\in C^2$ in a neighborhood of $x_0$, then $D^2 (u-\varphi)(x_0)\le 0$, i.e. $D^2u(x_0)\le D^2 \varphi(x_0)$, and, using the result of  Lemma \ref{2D-lem-degelliptic}, we'll obtain
    $$
    F(x_0, u(x_0), D\varphi (x_0), D^2\varphi(x_0))\le F(x_0, u(x_0), Du(x_0), D^2u(x_0))=0.
    $$

\item{$x_0\in\{u=0\}$}. Then, as in the previous case, $u$ is a subsolution for $-\Delta v -\lambda^-=0$. Now if $x_0$ is a local maximum point for $u-\varphi$ for some $\varphi\in C^2$, then
$$-\Delta\varphi(x_0)-\lambda^-(x_0)\le 0.$$
Hence,
\begin{multline*}
F(x_0, u(x_0), D\varphi (x_0), D^2\varphi(x_0))=\\
\min(-\Delta \varphi (x_0) +\lambda^+(x_0),\max(-\Delta \varphi(x_0)-\lambda^-(x_0),u(x_0)))=\\
\min(-\Delta \varphi (x_0) +\lambda^+(x_0),\max(-\Delta \varphi(x_0)-\lambda^-(x_0),0))=\\
\min(-\Delta \varphi (x_0) +\lambda^+(x_0),0)\le 0.
\end{multline*}
\end{itemize}
Thus, we have proved that $u$ is a viscosity subsolution for \eqref{2D-Problem-Viscosity}. Analogously we can obtain that $u$ is also a viscosity supersolution for \eqref{2D-Problem-Viscosity}.

For the proof that $u$ satisfies \eqref{2D-Problem-Viscosity} a.e. we refer to \cite{MR1825655}.
\end{proof}

\subsection{FDS, existence and uniqueness of discrete solution}

Now we are going to construct a Finite Difference Scheme (FDS) for one- and two-dimensional two-phase obstacle problems based on its  Min-Max form \eqref{2D-Problem-Viscosity}. For the sake of simplicity, we will assume that $\Omega=(-1,1)$ in one-dimensional case and $\Omega=(-1,1)\times(-1,1)$ in two-dimensional case in the rest of the paper, keeping in mind that the method works also for more complicated domains.

Let $N\in\mathbb{N}$ be a positive integer, $h=2/N$ and
$$
x_i=-1+ ih,\, y_i=-1+ih,\quad i=0,1,...,N.
$$

We are interested in computing approximate values of the two-phase obstacle problem solution at the grid points $x_i$ or $(x_i, y_j)$ in one- and two-dimensional cases, respectively. We will develop the one-dimensional and two-dimensional cases parallelly in this section, hoping that the same notations for this two cases will not make confusion for reader.  We use the notation $u_i$ and $u_{i,j}$ (or simply $u_\alpha$, where $\alpha$ is one- or two-dimensional multi-index) for finite-difference scheme approximation to $u(x_i)$ and $u(x_i,y_j)$, $\lambda^\pm_i=\lambda^\pm(x_i)$ and $\lambda^\pm_{i,j}=\lambda^\pm(x_i,y_j)$, $g_i=g(x_i)$ and $g_{i,j}=g(x_i, y_j)$ in one- and two-dimensional cases, respectively, assuming that the functions $g$ and $\lambda^\pm$ are extended to be zero everywhere outside the boundary $\partial\Omega$ and outside $\Omega$, respectively. In this section we will use also notations $u=(u_{\alpha})$, $g=(g_{\alpha})$ and $\lambda^\pm=(\lambda^\pm_{\alpha})$ (not to be confused with functions $u, g$ and $\lambda^\pm$). Also we will write $(a_{\alpha})\le (b_{\alpha})$ in $\mathcal I$ if $a_{\alpha}\le b_{\alpha}$ for all $\alpha\in\mathcal I$.

Denote
$$
 \mathcal N=\{i:\ 0\leq i\leq N\}\quad\mbox{or}\quad \mathcal N=\{(i,j):\ 0\leq i,j\leq N\},
$$
$$
  \mathcal N^o=\{i:\ 1\leq i\leq N-1\} \quad\mbox{or}\quad  \mathcal N^o=\{(i,j):\ 1\leq i,j\leq N-1\},
$$
in one- and two- dimensional cases, respectively, and
$$
  \partial \mathcal N=\mathcal N \setminus \mathcal N^o.
$$

In one-dimensional case we consider the following approximation for Laplace operator: for any $i\in \mathcal N^o$,
$$
  L_h u_{i}=\frac{u_{i-1}-2u_{i}+u_{i+1}}{h^2},
$$
and for two-dimensional case we introduce the following 5-point stencil approximation for Laplacian:
$$
  L_h u_{i,j}=\frac{u_{i-1,j}+u_{i+1,j}-4u_{i,j}+u_{i,j-1}+u_{i,j+1}}{h^2}
$$
for any $(i,j)\in \mathcal N^o$.

Applying the finite difference method to \eqref{2D-Problem-Viscosity}, we obtain the following nonlinear system:

\begin{equation}\label{nonlinprob}
\begin{cases}
 \min(-L_h u_{\alpha} + \lambda^+_{\alpha}\, ,\, \max(-L_h u_{\alpha}-\lambda^-_{\alpha}\, ,\,  u_{\alpha}))=0, &\alpha\in \mathcal N^o,\\[6pt]
 u_{\alpha}=g_{\alpha}, &\alpha \in \partial \mathcal N.
\end{cases}
\end{equation}

It is not clear a priori, whether this system has a solution, or, in the case of existence, this solution is unique. To this end, we consider the following functional:
$$
J_h(v)=-\frac{1}{2} \Big(L_h v,v\Big)+\Big(\lambda^+, v \vee 0\Big)-\Big(\lambda^-,v\wedge 0\Big)-\Big(L_h g,v\Big),
$$
defined on the finite dimensional space
$$
\mathcal K=\{v\in\mathcal H: \ v_{\alpha}=0, \ \alpha\in\partial \mathcal N\},\quad {\rm where}\quad
\mathcal H=\{v=(v_\alpha): v_\alpha\in \mathbb R, \ \alpha\in \mathcal N\}.
$$
Here $v\vee 0=\max(v,0)$, $v\wedge 0=\min(v,0)$ and for $w=(w_\alpha)$ and $v=(v_\alpha)$, $\alpha\in \mathcal N$, the inner product $(\cdot, \cdot)$ is defined by
$$
  (w,v)=\sum_{\alpha\in \mathcal N}w_\alpha\cdot v_\alpha.
$$

\begin{lemma}
The element $u\in\mathcal H$ solves \eqref{nonlinprob} if and only if $\tilde u=u-g$ solves the following minimization problem:
\begin{equation}\label{minprob}
    \tilde u\in \mathcal K:\qquad J_h(\tilde u)=\min_{v\in \mathcal K}J_h(v).
\end{equation}
\end{lemma}

\begin{proof}
Suppose $ \tilde u\in \mathcal K$ solves \eqref{minprob}. We choose arbitrary $w=(w_\alpha)\in\mathcal K$ and $t>0$, and denote $v=\tilde u+tw$. Obviously,  $ v\in \mathcal K$. It follows that
\begin{multline*}
 J_h(v)-J_h(\tilde u)= -\frac{t^2}{2}(L_hw,w)-t(L_h(\tilde u+g),w)+\\+(\lambda^+,(\tilde u+tw)\vee 0-\tilde u \vee 0)-(\lambda^-,(\tilde u+tw)\wedge 0-\tilde u\wedge 0)\geq 0.
\end{multline*}

Now, since $t$ is arbitrary positive number, we can conclude that
\begin{equation}\label{2P-eq-Lemma1.2-1}
  -t(L_h u,w)+(\lambda^+,(\tilde u+tw)\vee 0 - \tilde u \vee 0)-(\lambda^-,(\tilde u+tw)\wedge 0-\tilde u\wedge 0)\geq 0,
\end{equation}
if $t>0$ is sufficiently small.

To prove that $u$ satisfies \eqref{nonlinprob}, we treat several cases. First assume that $u_{\alpha_0}<0$ for some  ${\alpha_0}\in \mathcal N^o$.

By taking $w_{\alpha_0}=u_{\alpha_0}=\tilde u_{\alpha_0}$ and $w_\alpha=0$ for $\alpha\neq{\alpha_0}$ and substituting into \eqref{2P-eq-Lemma1.2-1}, we'll obtain
$$
  (-L_hu_{\alpha_0}-\lambda_{\alpha_0}^-)u_{\alpha_0}\geq 0.
$$
Now if we take $w_{\alpha_0}=-u_{\alpha_0}=-\tilde u_{\alpha_0}$ and  $w_\alpha=0$ for $\alpha\neq{\alpha_0}$, we'll
get from \eqref{2P-eq-Lemma1.2-1}
that $(-L_hu_{\alpha_0}-\lambda_{\alpha_0}^-)u_{\alpha_0}\leq 0.$ Hence,
\begin{equation}\label{2P-eq-Lemma1.2-2}
  -L_hu_{\alpha_0}=\lambda_{\alpha_0}^- , \quad \mbox{if} \quad u_{\alpha_0}<0.
\end{equation}
In the same way we can prove that
\begin{equation}\label{2P-eq-Lemma1.2-3}
  -L_hu_{\alpha_0}=-\lambda_{\alpha_0}^+ , \quad \mbox{if} \quad u_{\alpha_0}>0.
\end{equation}
Next we show that if $u_{\alpha_0}=0$ for some  ${\alpha_0}\in \mathcal N^o$, then
\begin{equation}\label{2P-eq-Lemma1.2-4}
-\lambda_{\alpha_0}^+\leq -L_hu_{\alpha_0}\leq \lambda_{\alpha_0}^-.
\end{equation}
Clearly, if we take in \eqref{2P-eq-Lemma1.2-1} $w_{\alpha_0}=1$ and $w_\alpha=0$, for $\alpha\neq{\alpha_0}$, we'll get
$  -L_hu_{\alpha_0}+\lambda_{\alpha_0}^+\geq 0,$ and if we take $w_{\alpha_0}=-1$ and  $w_\alpha=0$, for $\alpha\neq{\alpha_0}$, we'll get
$L_hu_{\alpha_0}+\lambda_{\alpha_0}^-\geq 0.$ Now, combining \eqref{2P-eq-Lemma1.2-2}, \eqref{2P-eq-Lemma1.2-3} and \eqref{2P-eq-Lemma1.2-4}, we conclude that $u$ satisfies \eqref{nonlinprob}.

Conversely, let $u\in \mathcal H$ satisfies \eqref{nonlinprob}. To prove that $\tilde u=u-g\in \mathcal K$ solves \eqref{minprob}, we take arbitrary $v\in \mathcal K$ and write
\begin{equation}\label{2P-eq-Lemma1.2-5}
  J_h(v)-J_h(\tilde u)=-\frac{1}{2}(L_h(v-\tilde u),v-\tilde u)-(L_h u,v-\tilde u)+(\lambda^+,v\vee 0-\tilde u\vee 0)-(\lambda^-,v\wedge 0 -\tilde u\wedge 0).
\end{equation}
It is well known fact that $\ -(L_h w, w)\geq 0$ for all $w\in\mathcal K$, so the first term in the right-hand side of \eqref{2P-eq-Lemma1.2-5} is nonnegative, and in order to prove our assertion, it is sufficient to prove that
\begin{equation}\label{eqrev}
-(L_h u,v-\tilde u)+(\lambda^+,v\vee 0-\tilde u\vee 0)-(\lambda^-,v\wedge 0 -\tilde u\wedge 0)\geq 0, \quad\forall v\in\mathcal K.
\end{equation}
To this end, we write
\begin{multline*}
-(L_h u,v-\tilde u)+(\lambda^+,v\vee 0-\tilde u\vee 0)-(\lambda^-,v\wedge 0 -\tilde u\wedge 0)= \\=\sum_{\alpha\in \mathcal N^o} \Big(-L_hu_\alpha\cdot (v_\alpha- u_\alpha)+\lambda^+_\alpha\cdot(v_\alpha\vee 0-u_\alpha\vee 0) - \lambda^-_\alpha\cdot(v_\alpha\wedge 0-u_\alpha\wedge 0)\Big)=\\ =\sum_{\{\alpha\in\mathcal N^o, u_\alpha<0\}} \Big(-L_hu_\alpha\cdot (v_\alpha- u_\alpha)+\lambda^+_\alpha\cdot(v_\alpha\vee 0-u_\alpha\vee 0) - \lambda^-_\alpha\cdot(v_\alpha\wedge 0-u_\alpha\wedge 0)\Big)+\\+\sum_{\{\alpha\in\mathcal N^o, u_\alpha>0\}} \Big(-L_hu_\alpha\cdot (v_\alpha- u_\alpha)+\lambda^+_\alpha\cdot(v_\alpha\vee 0-u_\alpha\vee 0) - \lambda^-_\alpha\cdot(v_\alpha\wedge 0-u_\alpha\wedge 0)\Big)+\\+\sum_{\{\alpha\in\mathcal N^o, u_\alpha=0\}} \Big(-L_hu_\alpha\cdot (v_\alpha- u_\alpha)+\lambda^+_\alpha\cdot(v_\alpha\vee 0-u_\alpha\vee 0) - \lambda^-_\alpha\cdot(v_\alpha\wedge 0-u_\alpha\wedge 0)\Big).
\end{multline*}

Since $u$ satisfies \eqref{nonlinprob}, we have
\begin{eqnarray*}
  -L_hu_\alpha=\lambda_\alpha^-, \quad &\text{when}& \quad u_\alpha<0, \\
  -L_hu_\alpha=-\lambda_\alpha^+, \quad &\text{when}& \quad u_\alpha>0, \\
  -\lambda_\alpha^+\leq -L_hu_\alpha\leq \lambda_\alpha^-, \quad &\text{when}& \quad  u_\alpha=0.
\end{eqnarray*}
Consequently,
\begin{multline*}
\sum_{\{\alpha\in\mathcal N^o, u_\alpha<0\}} \Big(-L_hu_\alpha\cdot (v_\alpha- u_\alpha)+\lambda^+_\alpha\cdot(v_\alpha\vee 0-u_\alpha\vee 0) - \lambda^-_\alpha\cdot(v_\alpha\wedge 0-u_\alpha\wedge 0)\Big)=\\=\sum_{\{\alpha\in\mathcal N^o, u_\alpha<0\}} \Big(\lambda_\alpha^-\cdot (v_\alpha- u_\alpha)+\lambda^+_\alpha\cdot (v_\alpha\vee 0) - \lambda^-_\alpha\cdot(v_\alpha\wedge 0-u_\alpha)\Big)=\\=
\sum_{\{\alpha\in\mathcal N^o, u_\alpha<0\}} \Big(\lambda_\alpha^-\cdot (v_\alpha- v_\alpha\wedge 0)+\lambda^+_\alpha\cdot (v_\alpha\vee 0)\Big)\ge 0,
\end{multline*}
\begin{multline*}
\sum_{\{\alpha\in\mathcal N^o, u_\alpha>0\}} \Big(-L_hu_\alpha\cdot (v_\alpha- u_\alpha)+\lambda^+_\alpha\cdot(v_\alpha\vee 0-u_\alpha\vee 0) - \lambda^-_\alpha\cdot(v_\alpha\wedge 0-u_\alpha\wedge 0)\Big)=\\=\sum_{\{\alpha\in\mathcal N^o, u_\alpha>0\}} \Big(-\lambda^+_\alpha\cdot (v_\alpha- u_\alpha)+\lambda^+_\alpha\cdot(v_\alpha\vee 0-u_\alpha) - \lambda^-_\alpha\cdot(v_\alpha\wedge 0)\Big)=\\=\sum_{\{\alpha\in\mathcal N^o, u_\alpha>0\}} \Big(\lambda^+_\alpha\cdot (v_\alpha\vee 0-v_\alpha)- \lambda^-_\alpha\cdot(v_\alpha\wedge 0)\Big)\ge 0
\end{multline*}
and
\begin{multline*}
\sum_{\{\alpha\in\mathcal N^o, u_\alpha=0\}} \Big(-L_hu_\alpha\cdot (v_\alpha- u_\alpha)+\lambda^+_\alpha\cdot(v_\alpha\vee 0-u_\alpha\vee 0) - \lambda^-_\alpha\cdot(v_\alpha\wedge 0-u_\alpha\wedge 0)\Big)=\\=\sum_{\{\alpha\in\mathcal N^o, u_\alpha=0\}} \Big(-L_hu_\alpha\cdot (v_\alpha\vee 0+v_\alpha\wedge 0) + \lambda^+_\alpha\cdot(v_\alpha\vee 0) - \lambda^-_\alpha\cdot(v_\alpha\wedge 0)\Big)=\\=\sum_{\{\alpha\in\mathcal N^o, u_\alpha=0\}} \Big(\big( -L_hu_\alpha + \lambda^+_\alpha\big)\cdot(v_\alpha\vee 0) +\big( -L_hu_\alpha- \lambda^-_\alpha\big)\cdot(v_\alpha\wedge 0)\Big)\ge 0.
\end{multline*}
 This completes the proof of the lemma.
\end{proof}

\begin{lemma}
The nonlinear system \eqref{nonlinprob} has a unique solution.
\end{lemma}

\begin{proof}
The minimization problem \eqref{minprob} has a unique solution, implying the existence of a unique solution to  \eqref{nonlinprob}.
\end{proof}

\subsection{Comparison principles for continuous and discrete nonlinear systems}

\begin{lemma}\label{1P-lem-comparison_continuous}
Let  $\Omega$ be a bounded domain and $v_1,v_2\in W^{2,\infty}(\Omega)$. If
$$
 \mathcal F [v_1] \le \mathcal F [v_2]\quad \mbox{a.e. in} \quad \Omega \qquad  \mbox{and} \qquad v_1\le v_2\quad \mbox{on} \quad \partial\Omega,
$$
then
$v_1\le v_2$ in $\Omega.$
\end{lemma}
\begin{proof} Let $\Omega_1=\{x\in\Omega:\mbox{  } v_1(x)>v_2(x)\}$.  If the set $\Omega_2=\{x\in \Omega_1: -\Delta v_1(x)>-\Delta v_2(x)\}$ has positive Lebesgue measure, then we get a contradiction, since
$\mathcal F [v_1](x)> \mathcal F [v_2](x)\quad \mbox{in}\quad \Omega_2.$ Consequently, $-\Delta v_1(x)\leq-\Delta v_2(x)$ a.e. in $\Omega_1.$ But in this case the weak maximum principle implies $v_2\geq v_1$ in $\Omega_1$, which is inconsistent with the definition of $\Omega_1$. Therefore, $\Omega_1=\emptyset$.
\end{proof}

To formulate the discrete analogue of the previous Lemma, we introduce the following notation:
$$
\Delta_h v(x)=\frac{v(x-h)-2v(x)+v(x+h)}{h^2},
$$
$$
\Delta_h v(x,y)=\frac{v(x-h,y)+v(x+h,y)+v(x,y-h)+v(x,y+h)-4v(x,y)}{h^2}
$$
in one- and two-dimensional cases, respectively, and
$$
\mathcal F_h[v]=\min\left(-\Delta_h v + \lambda^+ ,\max(-\Delta_h v-\lambda^-,v)\right), \quad x\in\Omega_h
$$
with $\Omega_h=\{\alpha\cdot h: \alpha\in\mathcal N^o\}$. Let also $\partial \Omega_h=\{\alpha\cdot h: \alpha\in\partial \mathcal N\}.$

\begin{lemma} \label{1P-lem-comparison_discrete}
Suppose $v_1,v_2\in \mathcal H.$ If
$$
\mathcal F_h [v_1]\le \mathcal F_h [v_2] \quad \mbox{in}\quad \Omega_h \qquad \mbox{and} \qquad v_1\le v_2\quad on \quad \partial\Omega_h,
$$
then $v_1\le v_2$ in $\Omega_h$.
\end{lemma}
\begin{proof} For the proof we refer to \cite{MR2218974}, where the author proves the comparison principle for more general type of schemes called degenerate elliptic schemes.
\end{proof}

\subsection{Regularization and error estimate}

The technique developed in this section applies for any dimension $n$. The idea comes from \cite{MR1759507} and \cite{MR2532350}, where in the first article the author obtains some estimates for the rate of convergence of finite difference approximation for degenerate parabolic Bellman's equations, and in the second paper the method is developed to obtain the optimal convergence rate for finite difference approximation to American Option valuation problem.

%{\color{red}
%To obtain error estimate further we use Schauder type estimates. First possibility is to use classical results in case of smooth domains. To have corresponding estimates for rectangle we should add some natural conditions.
%At the edges of rectangle we should have following
%\begin{equation}\label{VolkovConditions}
%\begin{cases}
%g_{x_1x_1}(x_e)+g_{x_2x_2}(x_e)=\lambda^+(x_e) & \text{ if } g(x_e)>0,\\
%g_{x_1x_1}(x_e)+g_{x_2x_2}(x_e)=-\lambda^-(x_e) & \text{ if } g(x_e)<0,\\
%g_{x_1x_1}(x_e)+g_{x_2x_2}(x_e)=\gamma(\lambda^+(x_e)-\lambda^-(x_e)) & \text{ if } g(x_e)=0.\\
%\end{cases}
%\end{equation}
%where $x_e$ is an edge of rectangle and $\gamma\in(0,1)$ is a fixed real number.
%}

Let $\beta\in C^\infty(\mathbb R)$ be a function satisfying

$$
\beta(z)=1,\quad z\ge 1; \qquad  \beta (z)= 0, \quad z\le -1;
$$
$$
\beta' (z)\ge 0,\quad z\in\mathbb R,
$$
where $\beta_\varepsilon(x)=\beta\left(\frac{x}{\varepsilon}\right)$, $x\in \mathbb R$. We denote by $u^\varepsilon$ the solution of the following auxiliary problem:
\begin{equation}\label{2P-eq-Regularization}
\left\{
\begin{array}{ll}
\Delta u^\varepsilon=\lambda^+\cdot\beta_\varepsilon (u^\varepsilon)-\lambda^-\cdot\beta_\varepsilon (-u^\varepsilon)& in\ \Omega,\\
u^\varepsilon = g & on\ \partial\Omega.
\end{array}
\right.
\end{equation}

\begin{lemma}\label{2P-lem-u-u_eps-est}
If $u$ is the solution of two-phase obstacle problem, and $u^{\varepsilon}$ is the regularized solution (i.e. the solution of \eqref{2P-eq-Regularization}), then
\[|u-u^{\varepsilon}|\leq\varepsilon.\]
\end{lemma}
\begin{proof} It follows from the definition of $u^\varepsilon$ that
$$
-\lambda^-\le\Delta u^{\varepsilon}\leq\lambda^+.
$$
Now, if  $u^{\varepsilon}\leq\varepsilon$, then
\begin{multline*}
\mathcal F[u^{\varepsilon}-\varepsilon]=\min(-\Delta u^{\varepsilon} + \lambda^+ ,\max(-\Delta u^{\varepsilon}-\lambda^-,u^{\varepsilon}-\varepsilon))=\\
\max(-\Delta u^{\varepsilon}-\lambda^-,u^{\varepsilon}-\varepsilon)\leq 0=\mathcal F[u]
\end{multline*}
As to the case $u^{\varepsilon}>\varepsilon$, we obviously get that $\Delta u^{\varepsilon}=\lambda^+$.
Therefore
\begin{align*}
\mathcal F[u^{\varepsilon}-\varepsilon]                           &=\min(-\Delta u^{\varepsilon} + \lambda^+ ,\max(-\Delta u^{\varepsilon}-\lambda^-,u^{\varepsilon}-\varepsilon))\\&=\min(0,\max(- \lambda^+-\lambda^-,u^{\varepsilon}-\varepsilon))=\min(0,u^{\varepsilon}-\varepsilon)=0=\mathcal F[u].
\end{align*}
Hence,
$$
\mathcal F[u^{\varepsilon}-\varepsilon]\leq \mathcal F[u] \quad\mbox{in}\quad \Omega.
$$

By Lemma \ref{1P-lem-comparison_continuous} we obtain
$$
u^{\varepsilon}-\varepsilon\leq u.
$$
In the same way, by considering the cases $u^{\varepsilon}\geq -\varepsilon$ and $u^{\varepsilon}<-\varepsilon$,
we will get
$ \mathcal F[u^{\varepsilon}+\varepsilon]\geq \mathcal F[u],$ and using again Lemma \ref{1P-lem-comparison_continuous} we obtain
$$
u^{\varepsilon}+\varepsilon\geq u.
$$
\end{proof}

\begin{lemma}\label{2P-less_epsilon}
If $u^{\varepsilon}$ is the solution of \eqref{2P-eq-Regularization}, then
$$
   \left|\mathcal F[u^{\varepsilon}]\right|\leq\varepsilon \quad \mbox{in}\quad \Omega.
$$
\end{lemma}
\begin{proof}
It is easy to see that $\mathcal F[u^{\varepsilon}]=0$ when $\left|u^{\varepsilon}\right|>\varepsilon.$

In the case $0\leq u^{\varepsilon}\le \varepsilon$   we have
$$
0\leq \mathcal F[u^{\varepsilon}]=\min(-\Delta u^{\varepsilon} + \lambda^+ ,\max(-\Delta u^{\varepsilon}-\lambda^-,u^{\varepsilon}))
=\min(-\Delta u^{\varepsilon} + \lambda^+,u^{\varepsilon})\leq u^{\varepsilon}\leq\varepsilon.
$$
Similarly, in the case  $-\varepsilon\leq u^{\varepsilon}<0$ we can prove that
$$
-\varepsilon\leq \mathcal F[u^{\varepsilon}]\leq 0.
$$
\end{proof}

\section{Convergence of the PGS algorithm}

\subsection{PGS algorithm for one-dimensional two-phase obstacle problem}

Now we propose an algorithm to construct an iterative sequence converging to the solution to nonlinear system \eqref{nonlinprob}. The idea is based on well-known PSOR (Projected Successive Over-Relaxation) method (see \cite{MR0298922}). We will call our algorithm \textit{Projected Gauss-Seidel (PGS)} method, since the main ingredient here is the Gauss-Seidel iteration combined with projection step. It should be mentioned here that the Gauss-Seidel method is a particular case of SOR algorithm.

For the sake of simplicity, we consider here only the one-dimensional case. Let $u=(u_0,u_1,...,u_N)$ be the solution of \eqref{nonlinprob} in one-dimensional case. In particular, $u_0=g_0$ and $u_N=g_N$. We will use the notation $\tilde u=(u_1,u_2,...,u_{N-1})$. This is the unknown part of $u$ that needs to be calculated. If we introduce also the following $N-1$ dimensional vectors:
$$
\tilde{\lambda}^\pm=\left(\lambda_1^\pm -\frac{g_0}{h^2},\ \lambda_2^\pm,\ ...,\ \lambda_{N-2}^\pm,\ \lambda_{N-1}^\pm-\frac{g_N}{h^2}\right),
$$
then, in one-dimensional case, the system \eqref{nonlinprob} can be rewritten it the following equivalent form :

\begin{equation}\label{2P-eq-1D-MatrixForm}
    \begin{cases}
      {\rm if}\ \tilde u_i>0,\ {\rm then}\ (A\tilde u)_i=\tilde \lambda^+_i,\\
      {\rm if}\ \tilde u_i<0,\ {\rm then}\ (A\tilde u)_i=-\tilde \lambda^-_i,\\
      -\tilde\lambda^-_i\le (A\tilde u)_i\le \tilde \lambda^+_i,\quad\forall i,
    \end{cases}
\end{equation}
where $A$ is the $(N-1)\times (N-1)$ dimensional tridiagonal matrix with $-2$'s on its main diagonal and $1$'s on two parallels, and $\tilde u\vee 0$ and $\tilde u \wedge 0$ are componentwise positive and negative parts of $\tilde u$, respectively.

We suggest the following algorithm to solve \eqref{2P-eq-1D-MatrixForm}:

Given the initial approximation
$$
  \tilde u^o=(\tilde u^o_1,\tilde u^o_2,...,\tilde u^o_{N-1}),
$$
for every $k=1,2,...$ and $1\le i\le N-1$ we denote
$$
  z_i^1=\frac{1}{2}\left(\tilde u_{i-1}^k+\tilde u_{i+1}^{k-1}-h^2\cdot \tilde \lambda^+_i\right) , \qquad
  z_i^2=\frac{1}{2}\left(\tilde u_{i-1}^k+\tilde u_{i+1}^{k-1}+h^2\cdot \tilde \lambda^-_i\right),
$$
with $\tilde u_0^{k}=\tilde u_N^k=0$ for all $k$.

Note that $z_i^1$ is the $k$-th step solution for $A\tilde u=\tilde \lambda^+$ by Gauss-Seidel method and $z_i^2$ is the $k$-th step solution for $A\tilde u=-\tilde\lambda^-$.

Then proceed as follows:

\begin{equation}\label{2P-eq-algorithm}
\begin{array}{ll}
\mbox{if}\ \  z_i^1\ge 0,&\quad {\rm then}\quad \tilde u_i^k=z_i^1;\\[6pt]
\mbox{if}\ \  z_i^2\le 0,&\quad {\rm then}\quad \tilde u_i^k=z_i^2;\\[6pt]
\mbox{if}\ \  z_i^1<0<z_i^2,&\quad {\rm then}\quad \tilde u_i^k=0.
\end{array}
\end{equation}

We will call the sequence $\tilde u^k=\left(\tilde u^k_1,\tilde u^k_2,...,\tilde u^k_{N-1}\right)$ constructed in this way \textit{the sequence obtained by PGS method}. The next section is devoted to the convergence analysis of this sequence.

\subsection{Convergence of the PGS algorithm}
$ $
\begin{theorem}\label{2P-thm-PGS-convergence}
The sequence $\tilde u^k$ converges and $\displaystyle\lim_{k\to\infty} \tilde u^k=\tilde u$.
\end{theorem}

\begin{proof} Denote
$$
\tilde u^{k,i}=\left(\tilde u^k_1,\tilde u^k_2,...,\tilde u^k_i, \tilde u^{k-1}_{i+1},...,\tilde u^{k-1}_{N-1}\right),\quad i=1,...,N-1,\quad k\in\mathbb N,
$$
$$
u^{k,i}=\left(0,\tilde u^k_1,\tilde u^k_2,...,\tilde u^k_i, \tilde u^{k-1}_{i+1},...,\tilde u^{k-1}_{N-1},0\right)\in\mathcal K,\quad i=1,...,N-1,\quad k\in\mathbb N
$$
and
$\mathcal J_p= J_h\left(u^{k,i}\right)$ for $p=(N-1)(k-1)+i$ with $i=1,...,N-1.$

The main idea is to prove that $\mathcal J_p$ decreases.

First let $p\not\in\{q(N-1): q\in\mathbb N\}$, i.e. $i\ne N-1$. Then
\begin{multline*}
\mathcal J_p-\mathcal J_{p+1}= J_h\left(u^{k,i}\right)- J_h\left(u^{k,i+1}\right)=-\frac{1}{2}\left(L_h \left(u^{k,i}-u^{k,i+1}\right),u^{k,i}-u^{k,i+1}\right)-\\
\left(L_h u^{k,i+1},u^{k,i}-u^{k,i+1}\right)+\left(\lambda^+,u^{k,i}\vee 0-u^{k,i+1}\vee 0\right)-\left(\lambda^-,u^{k,i}\wedge 0-u^{k,i+1}\wedge 0\right)-\\
\left(L_h g,u^{k,i}-u^{k,i+1}\right)
=\frac{1}{h^2}\left(\tilde u^{k-1}_{i+1}-\tilde u^{k}_{i+1}\right)^2
-\frac{\tilde u^k_i-2\tilde u^{k}_{i+1}+\tilde u^{k-1}_{i+2}}{h^2}\cdot \left(\tilde u^{k-1}_{i+1}-\tilde u^{k}_{i+1}\right)+\\ \lambda^+_{i+1}\cdot\left[\tilde u^{k-1}_{i+1}\vee 0-\tilde u^{k}_{i+1}\vee 0\right]-
-\lambda^-_{i+1}\cdot\left[\tilde u^{k-1}_{i+1}\wedge 0-\tilde u^{k}_{i+1}\wedge 0\right]-(L_h g)_{i+1}\cdot \left(\tilde u^{k-1}_{i+1}-\tilde u^{k}_{i+1}\right).
\end{multline*}

We continue by considering three cases:

\paragraph{Case 1: $\tilde u^k_{i+1}> 0$} It follows from \eqref{2P-eq-algorithm} that
$\frac{\tilde u^k_i-2\tilde u^{k}_{i+1}+\tilde u^{k-1}_{i+2}}{h^2}=\tilde\lambda^+_{i+1}.$ Hence,
$$
\mathcal J_p-\mathcal J_{p+1}=\frac{1}{h^2}\left(\tilde u^{k-1}_{i+1}-\tilde u^{k}_{i+1}\right)^2-\tilde\lambda^+_{i+1}\cdot \left(\tilde u^{k-1}_{i+1}-\tilde u^{k}_{i+1}\right)+\lambda^+_{i+1}\cdot\left[\tilde u^{k-1}_{i+1}\vee 0-\tilde u^{k}_{i+1}\right]-
$$
$$
-\lambda^-_{i+1}\cdot \tilde u^{k-1}_{i+1}\wedge 0-(L_h g)_{i+1}\cdot \left(\tilde u^{k-1}_{i+1}-\tilde u^{k}_{i+1}\right).
$$

Now, if $1\le i<N-1$, then $\tilde\lambda^+_{i+1}=\lambda^+_{i+1}$ and $(L_h g)_{i+1}=0$, so
$$
\mathcal J_p-\mathcal J_{p+1}=\frac{1}{h^2}\left(\tilde u^{k-1}_{i+1}-\tilde u^{k}_{i+1}\right)^2- (\lambda^+_{i+1}+\lambda^-_{i+1})\cdot \left(\tilde u^{k-1}_{i+1}\wedge 0\right)\ge\frac{1}{h^2}\left(\tilde u^{k-1}_{i+1}-\tilde u^{k}_{i+1}\right)^2
$$

If $i=N-1$, then $\tilde\lambda^+_{i+1}=\lambda^+_{i+1}-\frac{g_N}{h^2}$ and $(L_h g)_{i+1}=\frac{g_N}{h^2}$, so
\begin{multline*}
\mathcal J_p-\mathcal J_{p+1}=\frac{1}{h^2}\left(\tilde u^{k-1}_{i+1}-\tilde u^{k}_{i+1}\right)^2- (\lambda^+_{i+1}+\lambda^-_{i+1})\cdot \left(\tilde u^{k-1}_{i+1}\wedge 0\right)+\frac{g_N}{h^2}\tilde u^k_{i+1}\ge\\
\frac{1}{h^2}\left(\tilde u^{k-1}_{i+1}-\tilde u^{k}_{i+1}\right)^2.
\end{multline*}
Hence, in this case we have
\begin{equation}\label{2p-eq-thm3.1-1}
\mathcal J_p-\mathcal J_{p+1}\ge\frac{1}{h^2}\left(\tilde u^{k-1}_{i+1}-\tilde u^{k}_{i+1}\right)^2.
\end{equation}

\paragraph{Case 2: $\tilde u^k_{i+1}< 0$} Analogously to the previous case we can prove that \eqref{2p-eq-thm3.1-1} holds also in this case.

\paragraph{Case 3: $\tilde u^k_{i+1}=0$} It follows from \eqref{2P-eq-algorithm} that either
$$
\frac{\tilde u^k_i+\tilde u^{k-1}_{i+1}}{h^2}=\tilde\lambda^+_{i+1}\quad \mbox{or}\quad \frac{\tilde u^k_i+\tilde u^{k-1}_{i+1}}{h^2}=-\tilde\lambda^-_{i+1}
$$
or
$$
\frac{\tilde u^k_i+\tilde u^{k-1}_{i+1}}{h^2}-\tilde\lambda^+_{i+1}< 0<\frac{\tilde u^k_i+\tilde u^{k-1}_{i+1}}{h^2}+\tilde\lambda^-_{i+1},
$$
depending on the signs of $z_{i+1}^1$ and $z_{i+1}^2$. The first two cases are treated analogously to the Cases 1 and 2, so we will consider only the third possibility. In that case

$$
\mathcal J_p-\mathcal J_{p+1}=\frac{1}{h^2}\left(\tilde u^{k-1}_{i+1}-\tilde u^{k}_{i+1}\right)^2-
\left(\tilde u^{k-1}_{i+1}\vee 0\right)\cdot\left( \frac{\tilde u^k_i+\tilde u^{k-1}_{i+1}}{h^2}-\tilde\lambda^+_{i+1}\right)-
$$
$$
-\left(\tilde u^{k-1}_{i+1}\wedge 0\right)\cdot\left( \frac{\tilde u^k_i+\tilde u^{k-1}_{i+1}}{h^2}+\tilde\lambda^-_{i+1}\right)-(L_h g)_{i+1}\cdot \tilde u^{k-1}_{i+1}.
$$

Now, treating, as above, the cases $1\le i <N-1$ and $i=N-1$ separately, we obtain that \eqref{2p-eq-thm3.1-1} holds also in this case.

So far we have considered the case $p\not\in\{q(N-1): q\in\mathbb N\}$. Now assume that $p\in\{q(N-1): q\in\mathbb N\}$. In that case we'll obtain
\begin{equation}\label{2p-eq-thm3.1-2}
\mathcal J_p-\mathcal J_{p+1}\ge \frac{1}{h^2}\left(\tilde u^{k}_{i+1}-\tilde u^{k-1}_{i+1}\right)^2.
\end{equation}

Summarizing, we deduce that $\mathcal J_p$ decreases, and, since it is also bounded from below, we obtain that the sequence $\mathcal J_p$ converges. But in that case from \eqref{2p-eq-thm3.1-1} and \eqref{2p-eq-thm3.1-2} we can conclude that $\tilde u^{k}_{i}$ is a Cauchy sequence, hence also converges for any fixed $i=1,...,N-1$.

Finally, it can be easily verified that the limit solves \eqref{nonlinprob}.
\end{proof}

\section{Numerical Examples}
\paragraph{Example 1}
 We consider the following one-dimensional two-phase obstacle problem:
$$
\left\{
\begin{array}{ll}
\Delta u = 8\cdot\chi_{\{u>0\}}-8\cdot\chi_{\{u<0\}}, & x\in (-1,1)\\
u(-1)=-1, \quad u(1)=1.
\end{array}
\right.
$$

In this case the exact solution can be written down as a piecewise polynomial function:
$$
  u(x)=\left\{
  \begin{array}{ll}
  4x^2-4x+1,& 0.5\le x\le 1,\\
  0,&  -0.5< x< 0.5,\\
  -4x^2-4x-1,& -1\le x\le -0.5.\\
  \end{array}
  \right.
$$

We use the above described discretization with $N=20$. The PGS algorithm
produces the result given in Figure~\ref{Fig-ProblemMP1}, and the
error between numerical and exact solution (after 10 and 20
iterations) is represented in Figure~\ref{Fig-ProblemMP1Err}.
    \begin{figure}[ht]
    \begin{minipage}[b]{0.49\linewidth}
    \begin{center}
    \includegraphics[scale=0.6]{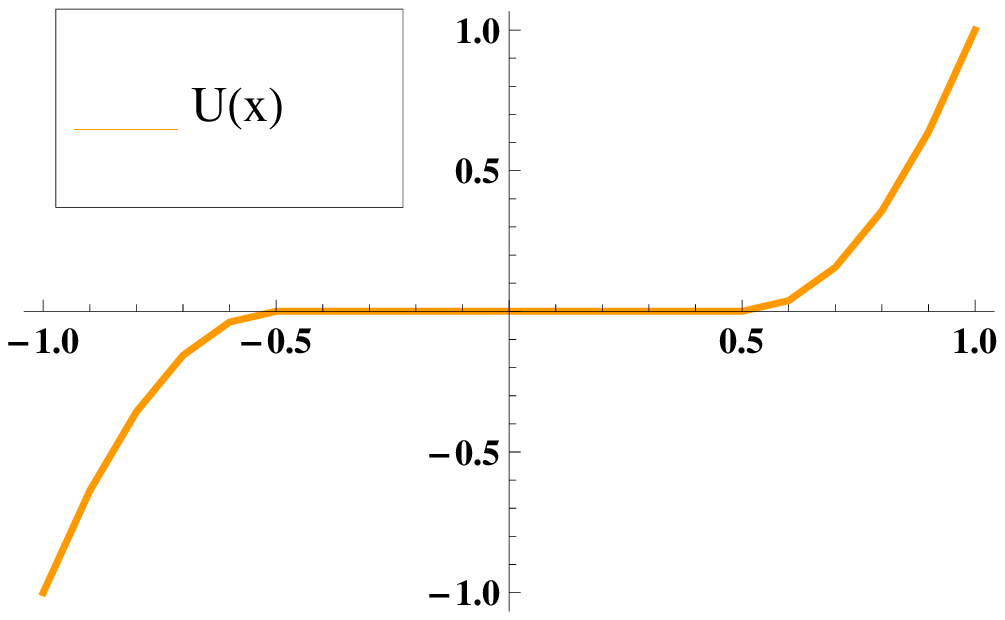}
    \end{center}
    \caption{Numerical Solution} \label{Fig-ProblemMP1}
    \end{minipage}
    \hspace{0.0cm}
    \begin{minipage}[b]{0.49\linewidth}
    \begin{center}
    \includegraphics[scale=0.6]{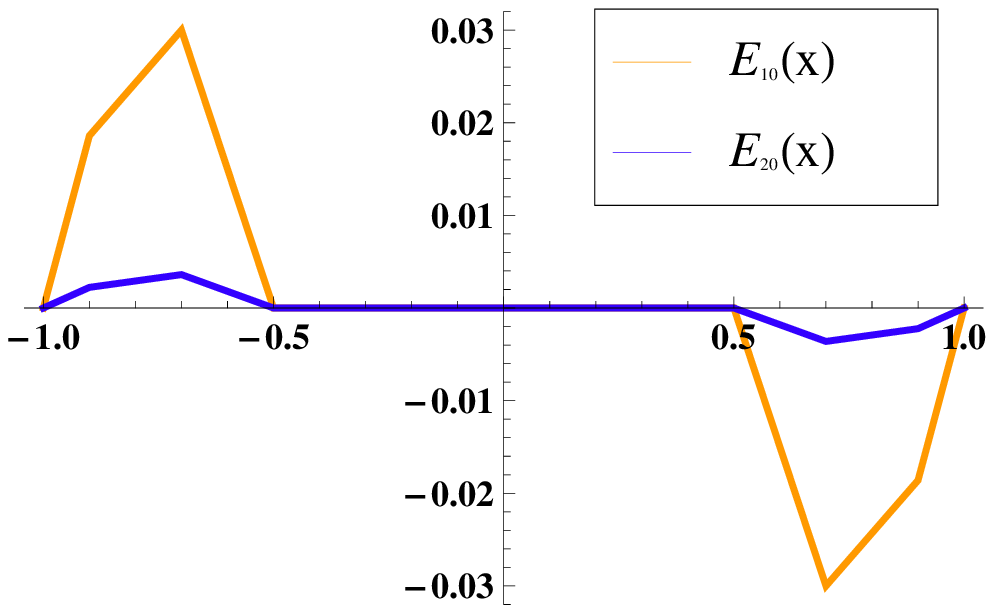}
    \end{center}
    \caption{Error between the exact and numerical solutions} \label{Fig-ProblemMP1Err}
    \end{minipage}
    \end{figure}

Next, the table \ref{tb1} we shows maximal errors between the exact and numerical solutions for this
example for different numbers of discretization points and iterations ($R_{N,M}$ is the maximal error while using $N$ discretization points and $M$ iterations). It is clearly visible that the error decrases along with the increase of $N$ and $M$.

\begin{center}
\begin{table}
    \caption{Error between the exact and numerical solutions}
    \label{tb1}
    \begin{tabular}{|c|c|c|c|c|c|}
    \hline
    &$N=20$ & $N=65$ & $N=120$ & $N=175$ & $N=230$\\
    \hline
    $R_{N, 2\times N}$&$0.0668629$ & $0.00236045$ & $0.005283$ & $0.0179411$ & $0.0347648$\\
    \hline
    $R_{N, 4\times N}$&$0.0668629$ & $0.00229779$ & $0.000577501$ & $0.00137638$ & $0.00445856$\\
    \hline
    $R_{N, 6\times N}$&$0.0668629$ & $0.0022977$ & $0.000556582$ & $0.000227299$ & $0.000658859$\\
    \hline
    $R_{N, 8\times N}$&$0.0668629$ & $0.0022977$ & $0.000556051$ & $0.000203333$ & $0.000134995$\\
    \hline
    $R_{N, 10\times N}$&$0.0668629$ & $0.0022977$ & $0.000556037$ & $0.00020249$ & $0.0000872308$\\
    \hline
    \end{tabular}
\end{table}
\end{center}

\paragraph{Example 2}  The second example is the following 2D two-phase problem:
\[
    \left\{
    \begin{array}{ll}
    \Delta u = 2\cdot\chi_{\{u>0\}}-2\cdot\chi_{\{u<0\}}, & (x,y)\in (-1,1)^2\\
    u(-1,y)=\left(\frac{1 - y}{2}\right)^2, \quad u(1,y)=\left(\frac{1 - y}{2}\right)^2,& y\in [-1,1]\\
    u(x,-1)=-x|x|,\quad u(x,1)=0, & x\in [-1,1].
    \end{array}
    \right.
\]
    \begin{figure}[ht]
    \begin{minipage}[b]{0.49\linewidth}
    \begin{center}
    \includegraphics[scale=0.59]{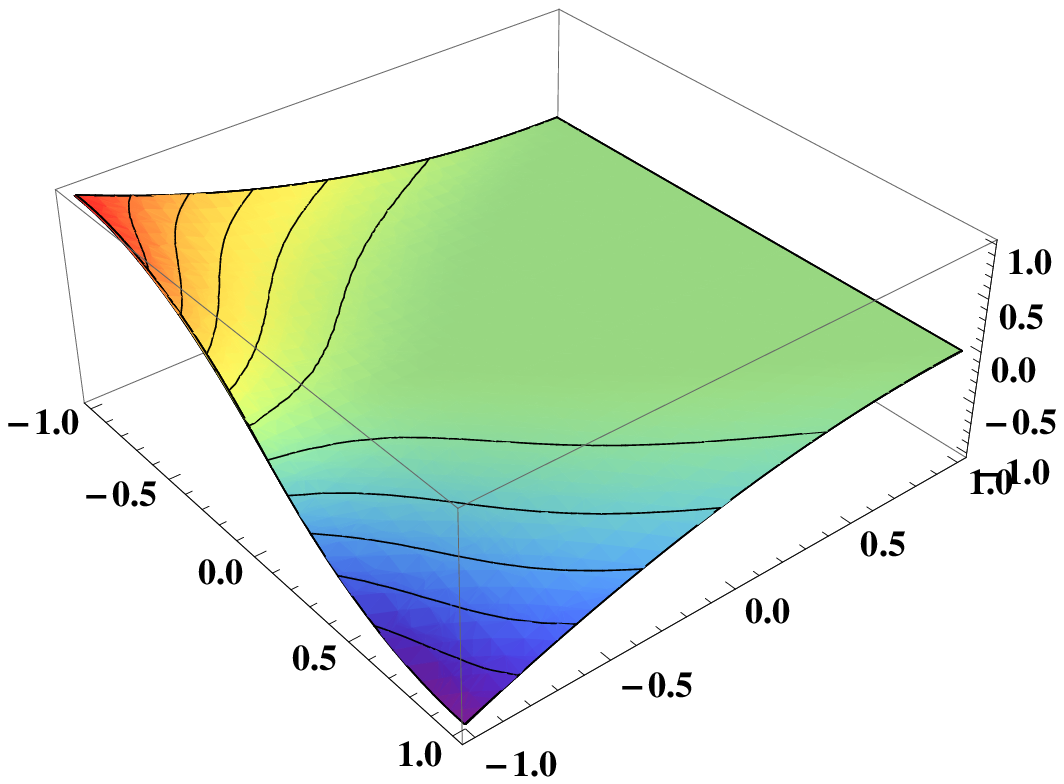}
    \end{center}
    \caption{Numerical Solution} \label{Fig-ProblemMP2}
    \end{minipage}
     \hspace{0.0cm}
    \begin{minipage}[b]{0.49\linewidth}
    \begin{center}
    \includegraphics[scale=0.45]{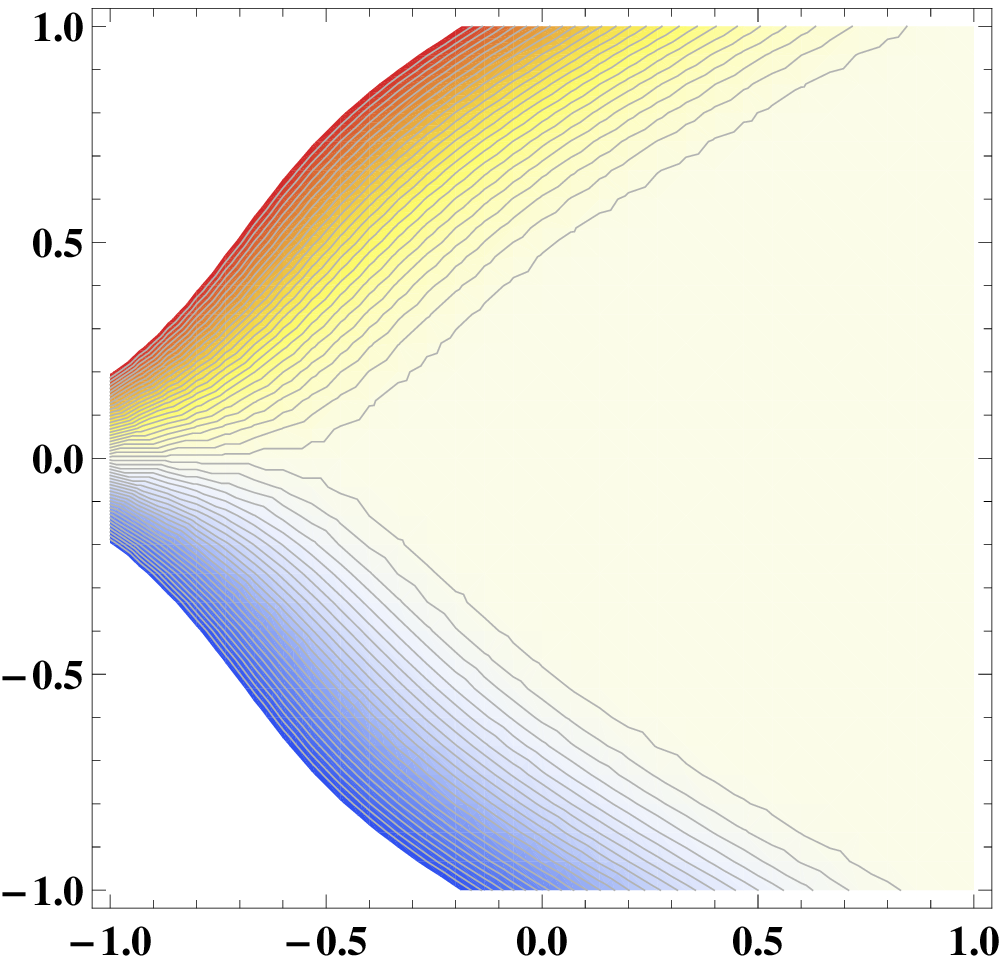}
    \end{center}
    \caption{Level sets} \label{Fig-ProblemMP2_level}
    \end{minipage}
    \end{figure}

    The numerical algorithm produces the result given in Figure
    \ref{Fig-ProblemMP2}:  the surface is the solution for our problem. Figure \ref{Fig-ProblemMP2}
    was constructed with 100 discretization  points and 400 iterations. The free boundary is clearly
    visible in Figure \ref{Fig-ProblemMP2_level} (the bell-shaped boundary of the
    white region, the zero-level set).

%\paragraph{Example 3} The third example is the two-phase problem with the following boundary data:
%    $$
%    \left\{
%    \begin{array}{ll}
%    \Delta u = 8\cdot\chi_{\{u>0\}}-8\cdot\chi_{\{u<0\}}, & (x,y)\in (-1,1)^2\\
%    u(-1,y)=1+y, \quad u(1,y)=1-y,    & y\in [-1,1]\\
%    u(x,-1)=1+x,\quad u(x,1)=1-x,     & x\in [-1,1].
%    \end{array}
%    \right.
%    $$
%The numerical algorithm produces the result given in Figure
%\ref{Fig-twophase_2}. It was constructed also with 100 discretization points and 400 iterations.
%
%
%    \begin{figure}[ht]
%    \begin{minipage}[b]{0.49\linewidth}
%    \begin{center}
%    \includegraphics[scale=0.59]{twophase2}
%    \end{center}
%    \caption{Numerical Solution}\label{Fig-twophase_2}
%    \end{minipage}
%    \hspace{0.0cm}
%    \begin{minipage}[b]{0.49\linewidth}
%    \begin{center}
%    \includegraphics[scale=0.45]{twophase2_contour}
%    \end{center}
%    \caption{Level Sets}\label{Fig-twophaselevel_2}
%    \end{minipage}
%    \end{figure}
%
%    It is important to mention that in Figure \ref{Fig-twophaselevel_2}, like in  Figure \ref{Fig-ProblemMP2_level},  the tangential touch of two branches of the free boundary is visible.

It is important to mention that in Figure \ref{Fig-ProblemMP2_level}  the tangential touch of two branches of the free boundary is clearly visible.

\bibliographystyle{model1-num-names}
\bibliography{twophase}

\end{document}